\documentclass[11pt]{amsart}

\usepackage{amsmath, amsthm, amssymb}
\usepackage{color}
\usepackage{kotex}
\usepackage{esint}
\usepackage{comment}

\allowdisplaybreaks

\newcommand{\R}{{\mathbb R}}

\newcommand{\Ric}{\mathrm{Ric}}

\theoremstyle{plain}
\newtheorem{theorem}{Theorem}[section]
\newtheorem{remark}[theorem]{Remark}

\newtheorem{corollary}[theorem]{Corollary}

\newtheorem{lemma}[theorem]{Lemma}

\numberwithin{equation}{section}

\title{Rigidity of the gradient estimate for Einstein manifolds}

\author[Lee]{Sanghoon Lee}
\address{Department of Mathematics  \\  The Chinese University of Hong Kong \\ Shatin, N.T., Hong Kong}
\email{sanghoonlee@cuhk.edu.hk}

\author[Park]{Jiewon Park}
\address{Korea Advanced Institute of Science and Technology (KAIST)\\
Department of Mathematical Sciences\\
Daejeon, South Korea}
\email{jiewonpark@kaist.ac.kr}

\begin{document}



\begin{abstract}
We study the rigidity of Ricci-flat manifolds with quadratic curvature decay under conditions on the Green function. We show that if the gradient of the Green function is uniformly bounded from below, then the manifold is flat. Furthermore, we prove that for a Ricci-flat manifold with quadratic curvature decay and Euclidean volume growth, the curvature is in $L^p$ for any $p \ge 2$. Combining with Cheeger-Tian \cite{CT} and Kr\"{o}ncke-Szab\'{o} \cite{KS}, we obtain that the manifold must be ALE of optimal order.
\end{abstract}

\maketitle

\section{Introduction}

We consider a complete nonparabolic Riemannian manifold $(M^n,g)$, $n\ge3$, and fix a point $p\in M$. Let $G(p,\cdot)$ denote the minimal positive Green function with pole at $p$. Following Colding \cite{Col}, we introduce the renormalized function
\[
b(x) := G(p,x)^{\frac{1}{2-n}}.
\]
A direct computation shows that $b$ satisfies the identity
\[
\Delta b^2 = 2n |\nabla b|^2.
\]
Since $b$ behaves like the distance function from $p$ (on $\R^n$ one has $b(x)=|x-p|$), it plays an important role in the study of rigidity phenomena and geometric inequalities on manifolds with Ricci curvature bounds. 

In \cite[Theorem 3.1]{Col}, Colding established the following sharp gradient estimate for $b$.

\begin{theorem}[{\cite[Theorem 3.1]{Col}}] \label{thm: Colding} If moreover $\Ric\ge 0$ on $M$, then
\begin{equation} \label{eq: grad est}
|\nabla b| \leq 1,    
\end{equation}
 with equality at one point implying that M is isometric to the flat Euclidean space $\R^n$.
\end{theorem}

A natural question is whether one can obtain rigidity under a complementary lower bound on $|\nabla b|$. In this paper we address this question under additional structural assumptions on $(M,g)$. Our first result shows that, for Ricci-flat manifolds with quadratic curvature decay, a uniform positive lower bound for $|\nabla b|$ forces flatness. Throughout we write $d(x,y)$ for the Riemannian distance between $x,y\in M$.

\begin{theorem}\label{thm: rigidity}
For each $\delta\in (0,1)$ and dimension $n\ge 3$, there exists a positive constant $C(n, \delta)$ with the following significance.

Let $(M, g)$ be an $n$-dimensional nonparabolic Ricci-flat Riemannian manifold and $p\in M$ a fixed point. Suppose that  
$$
|Rm|(x) \le \frac{C(n,\delta)}{d(p,x)^2}
\quad\text{and}\quad
|\nabla b|(x) \ge \delta
$$
for all $x \in M$. Then $M$ is flat.
\end{theorem}

\begin{remark}
The constant $C(n,\delta)$ can be explicitly computed, as our proofs do not rely on any compactness argument.
\end{remark}

When $M$ has Euclidean volume growth, i.e. when $v_M:=\lim_{r\to \infty} \frac{\mathrm{vol}(B_p(r))}{\omega_n r^n}>0$
where $\omega_n$ denotes the volume of the unit $n$-ball in $\R^n$, we obtain the following rigidity result by combining with works of Colding-Minicozzi \cite{CM, CM2}.

\begin{theorem} \label{thm: rigidity2}

For each dimension $n\ge 3$, there exists a positive constant $C(n)$ with the following significance.

Let $(M, g)$ be an $n$-dimensional nonparabolic Ricci-flat Riemannian manifold with Euclidean volume growth, and $p\in M$ a fixed point. Suppose that 
$$|Rm|(x) \le \frac{C(n)}{d(p,x)^2}$$  for all $x \in M$. Then $\int_M |Rm|^p < \infty$ for any $p \ge 2$.
\end{theorem}

Combining Theorem \ref{thm: rigidity2} with the results of Bando–Kasue–Nakajima \cite{BKN}, Cheeger–Tian \cite{CT}, and Kr\"{o}ncke–Szab\'{o} \cite{KS}, we immediately obtain the following corollary.

\begin{corollary} \label{cor: rigidity2}
Under the same assumptions as in Theorem \ref{thm: rigidity2}, $M \setminus K$ is diffeomorphic to $\left(\R^n \backslash B_0(R)\right)/\Gamma$ for some compact set $K \subset M$, where $B_0(R)$ is the ball of radius $R$ centered at $0$ in $\R^n$. Moreover, $M$ is an ALE manifold of order $n$.
\end{corollary}

The key ingredients in the proof of Theorem \ref{thm: rigidity} and Theorem \ref{thm: rigidity2} are monotonicity formulas tailored to Ricci-flat manifolds, given in \eqref{eqn:monotonicity1} and \eqref{eqn:monotonicity2}. Our methods are inspired by similar techniques used by Lee-Wang \cite{LW} in their study of Poincar\'{e}-Einstein manifolds with the flat Euclidean space as their conformal infinity.

Under the assumption of Euclidean volume growth,  it is known that the pointwise decay $\lim_{d(p,x)\to \infty}|Rm|(x) = 0$ already implies a quadratic curvature decay $|Rm|\le C/d^2$ for some $C>0$; see Cheng--Zhu \cite[Theorem 1.3]{CZ}. On the other hand, the lower bound assumption on $|\nabla b|$ in Theorem \ref{thm: rigidity} essentially cannot be omitted. The Eguchi-Hanson metric provides a complete non-flat Ricci-flat manifold with Euclidean volume growth and quadratic curvature decay. There is now an abundance of ALE and AC spaces, constructed for instance in \cite{Kro1, Kro2, St, TY}, which suggests that additional conditions beyond quadratic curvature decay are necessary to conclude flatness.

We conclude this introduction by mentioning that rigidity theorems for complete Einstein manifolds have been extensively studied. Regarding $L^{n/2}$ curvature bounds, Anderson \cite[Theorem 3.5]{And} proved that if $M$ is a complete Ricci-flat $n$-manifold of Euclidean volume growth and
\[
\int_M |Rm|^{n/2} < \infty,
\]
then $M^n \setminus B_p(R_0)$ is diffeomorphic to $(R_0,\infty) \times \mathbb{S}^{n-1}/\Gamma$ for some large $R_0>0$, where $\mathbb{S}^{n-1}/\Gamma$ is a space form. Shen \cite{Shen} showed that there is a constant $c(n)$ so that if moreover $$\int_M |Rm|^{n/2} \le c(n) v_M^{n+1},$$ then $M$ is isometric to $\R^n$. As for gap theorems on the volume, Anderson \cite{And2} showed that there exists $\epsilon(n)>0$ for which if the Ricci-flat manifold $M$ has $v_M > 1-\epsilon(n)$, then $M$  is isometric to $\R^n$. Honda-Mondino \cite{HM} confirmed that the optimal $\epsilon(4)$ is $1/2$; we refer the reader to \cite{HM} for the proof and several interesting open questions. We also remark that Yokota \cite{Yokota} proved an analogous gap theorem on gradient Ricci solitons for the Gaussian density. Our results are complementary to these works, focusing on bounds on the Green function.

This paper is organized as follows. In Section \ref{sec: ODE} we prove a model ODE lemma which captures the mechanism behind our main theorems. In Section \ref{sec: proof} we apply the same ideas to the curvature tensor on Ricci-flat manifolds and prove Theorems \ref{thm: rigidity} and \ref{thm: rigidity2}.
\\
\\
\noindent {\bf Acknowledgement.} JP was supported by the National Research Foundation of Korea (NRF) Grant RS-2024-00346651.

\section{An ODE lemma} \label{sec: ODE}

The following lemma on the rigidity of quadratically decaying solutions of a linear elliptic equation serves as a model problem for our main results.

\begin{lemma} \label{lem: ODE}
Let $n \ge 3$ and $f:\R^n \to \R$ be a smooth function with quadratic decay, so that $\underset{B_0(r)}{\sup}|f|\le\frac{C_1}{r^2}$ for any $r>0$ where $C_1\le (n-2)^2/8$ is a positive constant. Let $u:\R^n \to \R$ be a smooth solution to the equation \[\Delta u = - fu.\]
Then $g(r)$ defined by $\textstyle g(r) =  \int_0^r s^{n-3} \fint_{\partial B_s} u^2 \, ds$ enjoys the growth estimate
\begin{equation}
g(s) \le \frac{s^xg(r)}{r^x}
\end{equation}
for any $r\ge s>0$, where $x=\left(n-2+\sqrt{(n-2)^2-8C_1}\right)/2$.
Moreover, if $n<10$ or $C_1<2n-12$, then any quadratically decaying solution $u$ so that $\underset{B_0(r)}{\sup}|u|\le\frac{C_0}{r^2}$ for some constant $C_0>0$ must be identically zero.
\end{lemma}

\begin{proof}
Let $B_r$ denote a ball of radius $r$ centered at a fixed point. Integration by parts yields
\begin{align}\label{eq: ODE ineq}
\begin{split}
0\le\frac{1}{r^{n-1}} \int_{B_r} |Du|^2&=-\frac{1}{r^{n-1}}\int_{B_r}u\Delta u+\frac{1}{r^{n-1}}\int_{\partial B_r}u\partial_r u\\
&=\frac{1}{r^{n-1}} \int_{B_r} fu^2+\frac{1}{r^{n-1}}\int_{\partial B_r}u\partial_r u.
\end{split}
\end{align}

The first term can be estimated by
\begin{align*}
\int_{B_r} fu^2&\le\int_0^r \underset{\partial B_s}{\sup}|f|\cdot|\partial B_s|\fint_{\partial B_s}u^2 \, ds=(n+1)\omega_{n+1}\int_0^r \underset{\partial B_s}{\sup}|f|\cdot s^{n-1}\fint_{\partial B_s}u^2 \, ds\\
&\le (n+1)\omega_{n+1}C_1 \int_0^r s^{n-3}\fint_{\partial B_s}u^2 \, ds,
\end{align*}
where we used that $|\partial B_s|=(n+1)\omega_{n+1}s^{n-1}$, $\omega_{n+1}$ denoting the volume of the unit $(n+1)$-ball. Denote $g(r)$ by
\[g(r) =  \int_0^r s^{n-3} \fint_{\partial B_s} u^2 \, ds.\]
Then $\displaystyle \int_{B_r} fu^2\le (n+1)\omega_{n+1}C_1 g(r)$, and \[g'(r)=r^{n-3}\fint_{\partial B_r} u^2.\] Differentiating gives \[\displaystyle g''(r) = \frac{n-3}{r}g'(r)+2r^{n-3}\fint_{\partial B_r}u \partial_r u.\]
We may now rewrite \eqref{eq: ODE ineq} to be
\begin{align*}
\begin{split}
    0&\le\frac{1}{r^{n-1}}\int_{B_r}fu^2 +\frac{|\partial B_r|}{r^{n-1}}\fint_{\partial B_r}u\partial_r u\\
&\le \frac{(n+1)\omega_{n+1}C_1}{r^{n-1}} g(r)+\frac{(n+1)\omega_{n+1}}{2r^{n-3}}\left(g''(r)-\frac{n-3}{r}g'(r)\right).
\end{split}
\end{align*}
which is equivalent to the following inequality,
\begin{equation} \label{eq: ODE simplified}
\frac{(n-3)}{r} g'(r) \le g''(r) + \frac{2C_1}{r^2} g(r).
\end{equation}
Let $x\ge y$ be the indicial roots of the above homogeneous ODE, i.e. they solve the quadratic equation \[t^2 -(n-2)t + 2C_1 = 0.\]
Such $x, y$ exist if the inequality $(n-2)^2 \ge8C_1$ holds. They are both positive and less than $(n-2)$. 
Letting $h(r) =g(r)/r^x$, \eqref{eq: ODE simplified} is equivalent to
\begin{equation} \label{eqn:monotonicity1}
0 \le \left(\frac{h'(r)}{r^{y-x-1}}\right)'.
\end{equation}
The definition of $g(r)$ gives that $g(r)=O(r^{n-2})$ and $g'(r)=O(r^{n-3})$ for $r\ll 1$. Therefore 
\[\frac{h'(r)}{r^{y-x-1}}=\frac{rg'(r)-xg(r)}{r^{y}} \to 0 \text{ as }r\to 0.\]
Thus $\frac{h'(r)}{r^{y-x-1}} \ge 0$ and $h'(r) \ge 0$ for all $r\ge 0$. This implies for $r \ge s >0,$
$$g(s) \le \frac{s^xg(r)}{r^x}.$$
Now if $u$ also decays quadratically by $\underset{B_0(r)}{\sup}|u| \le C_0/r^2$, then $g(r)=O(r^{n-6})$ as $r\to \infty$ for $n \neq 6$, and $g(r)=O(\log r)$ as $r\to \infty$ for $n=6$. Thus, we can conclude that $g\equiv 0$ whenever $x=\frac{n-2}{2}+\sqrt{\left(\frac{n-2}{2}\right)^2-2C_1}>n-6$, which in turn implies that $u \equiv 0$.
\end{proof}

\section{Application to Ricci-flat manifolds} \label{sec: proof}

Let $(M^n,g)$ be a complete noncompact nonparabolic Riemannian manifold of dimension $n\ge 3$ with $\Ric\equiv 0$ and quadratic curvature decay, i.e. there exists a constant $C_0>0$ such that
\[
|Rm| \le \frac{C_0}{d^2},
\]
where $d$ denotes the distance from a fixed point $p \in M$.

Using the same notation as in the Introduction, let $G$ be the minimal positive Green function on $M$, and denote by $b:M \to \R$ the function
\[
b(x)=G(p,x)^{1/(2-n)}.
\]
Then the harmonicity of $G$ is equivalent to
\[\Delta b^2 = 2n|\nabla b|^2.\]

Arguing along the same lines as in Lemma \ref{lem: ODE}, we now prove our main results.

\begin{proof}[Proof of Theorem \ref{thm: rigidity}]
We recall the integral identity
\begin{equation} \label{eq: int identity}
\big(r^{1-n} \int_{\{b=r\}} f |\nabla b|\big)' = r^{1-n} \int_{\{b\le r\}} \Delta f    
\end{equation}
which can be found, for instance, in \cite{Col} and \cite{CM}. Define $f:M \to \R$ by
\[f=b^2 |Rm|^2.\]
Applying the above identity to $f$, we obtain
{\allowdisplaybreaks
\begin{align} \label{eqn: deriv ineq}
\begin{split}
\big(&r^{3-n} \int_{\{b=r\}} |Rm|^2 |\nabla b|\big)'=\big(r^{1-n}\int_{\{b=r\}} b^2|Rm|^2 |\nabla b|\big)'=r^{1-n} \int_{\{b\le r\}} \Delta(b^2|Rm|^2)\\
& =  r^{1-n} \int_{\{b\le r\}} \big[ b^2 \Delta|Rm|^2  + 2\nabla b^2 \cdot \nabla |Rm|^2 + |Rm|^2 \Delta b^2 \big] \\
&= r^{1-n} \int_{\{b\le r\}} \big[ b^2 \Delta|Rm|^2 - |Rm|^2 \Delta b^2 \big]+2r^{1-n}\int_{\{b\le r\}}\big[|Rm|^2\Delta b^2 + \nabla b^2 \cdot \nabla|Rm|^2\big]  \\
&=r^{1-n} \int_{\{b\le r\}} \big[ b^2 \Delta|Rm|^2 - |Rm|^2 \Delta b^2 \big]+2r^{1-n}\int_{\{b=r\}}|Rm|^2 \frac{\nabla b^2 \cdot \nabla b}{|\nabla b|} \\
& = r^{1-n} \int_{\{b\le r\}} \big[ b^2 \Delta|Rm|^2 - |Rm|^2 \Delta b^2 \big] +4 r^{2-n} \int_{\{ b =r \}} |Rm|^2|\nabla b| \\
& =  r^{1-n} \int_{\{b\le r\}} \big[ b^2 \Delta|Rm|^2 - 2n |Rm|^2 |\nabla b|^2 \big] +4 r^{2-n} \int_{\{ b =r \}} |Rm|^2|\nabla b| \\
& \ge r^{1-n}\int_{\{b\le r\}} \big[ -C(n) b^2 |Rm|^3 - 2n |Rm|^2 |\nabla b|^2 \big] +4 r^{2-n} \int_{\{ b =r \}} |Rm|^2|\nabla b|,
\end{split}
\end{align}}
where $C(n)>0$ is a dimensional constant; in the last line, we used the identity on Ricci-flat manifolds that $\Delta |Rm|^2 = 2|\nabla Rm|^2 + Rm*Rm*Rm$, hence $\Delta |Rm|^2 \ge -C(n)|Rm|^3.$

Next we observe that $b \leq d$ since $\lim_{d\to 0}b=0$ (see for example \cite[(2.19)]{CM}) and $|\nabla b| \leq 1$ by Theorem \ref{thm: Colding}. Hence, from the assumption that $|Rm|\le \frac{C_0}{d^2}$ for some $C_0>0$, we have $|Rm| \le\frac{C_0}{b^2}$. Together with the assumption that $|\nabla b| \ge \delta$, \eqref{eqn: deriv ineq} implies

\begin{align} \label{eq: deriv ineq 2}
\begin{split}
&\big( r^{3-n} \int_{\{b=r\}}|Rm|^2 |\nabla b|\big)' \\
&\ge -r^{1-n}\int_{\{b\le r\}} \big[\frac{C_0 C(n)}{\delta^2}  |Rm|^2 |\nabla b|^2+2n |Rm|^2 |\nabla b|^2 \big] +4 r^{2-n} \int_{\{ b =r \}} |Rm|^2|\nabla b|.
\end{split}
\end{align}

Define $F:[0,\infty)\to \R$ by
\[F(r) = \int_{\{b \le r\}} |Rm|^2|\nabla b|^2.\]
The coarea formula implies
\[F'(r) = \int_{\{b = r\}} |Rm|^2|\nabla b|.\]
Thus, \eqref{eq: deriv ineq 2} can be rewritten in terms of $F$ to be
\[\left(r^{3-n}F'(r)\right)' \ge -r^{1-n}\left(2n+\frac{C_0 C(n)}{\delta^2}\right)F(r)+4r^{2-n}F'(r),\]
or equivalently
\begin{equation} \label{eq: ineq for F}
F''(r)- \frac{(n+1)}{r} F'(r) + \frac{1}{r^2}\left(2n + \frac{C_0C(n)}{\delta^2}\right) F(r)  \ge 0.
\end{equation}

Let $x\ge y$ be the indicial roots of the homogeneous ODE, i.e. they solve the equation
\[t^2-(n+2)t+\left(2n + \frac{C_0C(n)}{\delta^2}\right)= 0.\]
If $C_0 <\frac{(n-2)^2\delta^2}{4C(n)}$, there are two positive roots; for small $C_0$, $x$ is close to $n$ and $y$ is close to $2$. In particular, $y \le (n+2)/2 < n$.

We observe that defining $h(r)=F(r)/r^x$, \eqref{eq: ineq for F} is equivalent to
\[\left(\frac{h'(r)}{r^{y-x-1}}\right)' \ge 0.\]
It follows from the definition of $F$ that  $F(r)=O(r^n)$ and $F'(r)=O(r^{n-1})$ for $r\ll 1$. Therefore
\begin{equation}\label{eqn:monotonicity2}
\frac{h'(r)}{r^{y-x-1}}=\frac{rF'(r)-xF(r)}{r^y}\to 0 \text{ as } r \to 0.
\end{equation}
Thus $h'(r)\ge 0$ for $r>0$, and
\[F(s) \le s^x \cdot \frac{F(r)}{r^x}\]

whenever $r\ge s>0$. On the other hand, in fact since $|\nabla b|\le 1$,
\[F(r)=\int_{\{b\le r\}}|Rm|^2|\nabla b|^2 \le C_0^2 r^{-4} \cdot \text{vol}(B_p(r))\le Cr^{n-4}\]
by the Bishop-Gromov inequality, where $C=C(C_0, n)$ is a constant. Altogether we have
\[F(s)\le Cs^x \cdot r^{n-4-x}.\]

When $C_0$ is so small that $x=\frac{n+2}{2}+\sqrt{\left(\frac{n-2}{2}\right)^2-\frac{C_0 C(n)}{\delta^2}}>n-4$, we can take the limit $r \rightarrow \infty$ to see that $F(s)=0$ for any $s>0$, thus $|Rm| \equiv 0$.
\end{proof}

\begin{proof}[Proof of Theorem \ref{thm: rigidity2}]
When the Ricci-flat manifold $M$ has Euclidean volume growth, by \cite[Section 2.3]{CM2} the limit $\lim_{b \rightarrow \infty} |\nabla b|=\lim_{b \rightarrow \infty}b/d=:\theta$ exists, where $\theta \in (0,1)$ is determined by the volume growth of $M$. Thus we can choose $r_0$ large so that $|\nabla b| \ge \theta/2$ and  $b/d\le 2\theta$ whenever $b \ge r_0$.

Since $|Rm|$ fails to be smooth where $Rm=0$, we take the standard regularization $u_\delta:=(|Rm|^2+\delta^2)^{1/2}$ of $|Rm|$ where $0<\delta<1$. By the improved Kato inequality for Ricci-flat manifolds (\cite[Corollary 4.10]{BKN}), there exist constants $C=C(n)$ and $\epsilon=\epsilon(n)$ so that $\Delta |Rm|^{1-\epsilon} \ge -  C|Rm|^{2-\epsilon}$ at points where $|Rm|\ne 0$, from which it follows that $\Delta (u_\delta)^{1-\epsilon}\ge -C|Rm|(u_\delta)^{1-\epsilon}$. On the other hand, the integral identity \eqref{eq: int identity} for $u_\delta$ yields 
\[\big(r^{1-n} \int_{\{b=r\}} (u_\delta)^{1-\epsilon}|\nabla b|\big)' = r^{1-n} \int_{\{b\le r\}} \Delta (u_\delta)^{1-\epsilon}.\]
Using these facts we compute for $r\ge r_0$,
\begin{align*}
&\big[r^{n-1}\big(r^{3-n}\int_{\{b=r\}} (u_\delta)^{1-\epsilon}|\nabla b|\big)'\big]'  \\
=&\int_{\{b =  r\}} \big[ b^2 \Delta(u_\delta)^{1-\epsilon} /|\nabla b| - 2n(u_\delta)^{1-\epsilon} |\nabla b| \big]+4\big( r \int_{\{ b =r \}} (u_\delta)^{1-\epsilon}|\nabla b| \big)'\\
&\ge -\int_{\{b =  r\}} \big[Cb^2|Rm|(u_\delta)^{1-\epsilon} /|\nabla b|+2n(u_\delta)^{1-\epsilon} |\nabla b| \big]+4\big( r \int_{\{ b =r \}} (u_\delta)^{1-\epsilon}|\nabla b| \big)'\\
&\ge -\int_{\{b =  r\}} \big[16CC_0(u_\delta)^{1-\epsilon}|\nabla b|+2n(u_\delta)^{1-\epsilon} |\nabla b| \big]+4\big( r \int_{\{ b =r \}} (u_\delta)^{1-\epsilon}|\nabla b| \big)'.
\end{align*}

Define \[F_\delta(r) = \int_{\{ b =r \}} (u_\delta)^{1-\epsilon}|\nabla b|.\] Then the differential inequality above can be written as
$$[r^{n-1}(r^{3-n} F_\delta(r))']' \ge  (-2n- 16CC_0) F_\delta(r) + 4(rF_\delta(r))'$$
which is equivalent to
\begin{equation}\label{eq: ineq for F_2}
r^2 F_\delta''(r) - (n-1)rF_\delta'(r) + ( n-1 +16CC_0)F_\delta(r) \ge 0.
\end{equation}
If $C_0$ is small (precisely, $C_0 < (n-2)^2/64C$), then the indicial roots $x > y$ exist and are positive, with $1 < y<x < n-1$. Arguing as in the proof of Theorem \ref{thm: rigidity}, we see that \eqref{eq: ineq for F_2} implies $0 \le \big( g_\delta'(r)/r^{y-x-1} \big)'$ where $g_\delta(r) := F_\delta(r)/r^x$. Hence, for $r \ge r_0$ we have
\[\frac{g_\delta'(r_0)}{r_0^{y-x-1} r^{x-y+1}} \le g_\delta'(r).\]
Integrating on $(r,\delta^{-1/2})$, we get
\[g_\delta(r) \le g_\delta(\delta^{-1/2}) + \frac{g_\delta'(r_0)}{(x-y)r_0^{y-x-1}}\left(\delta^{\frac{x-y}{2}}-\frac{1}{r^{x-y}}\right).\]

Using \eqref{eq: int identity} again, it is straightforward to calculate that
\[g_\delta'(r_0)=(n-1-x)r_0^{-x-1}\int_{\{b=r_0\}}u_\delta^{1-\epsilon}|\nabla b|+r_0^{-x}\int_{\{b\le r_0\}}\Delta (u_\delta)^{1-\epsilon}|\nabla b|,\]
from which it follows there exists $C(n,r_0)$ so that $|g_\delta'(r_0)|<C(n,r_0)$ for $\delta>0$ small. Moreover, $\lim_{\delta\to 0}g_\delta(\delta^{-1/2})=0$ if $C_0$ is so small that $x>n-1-2(1-\epsilon)$. Thus taking $\delta \to 0$ we obtain
\begin{equation} \label{eq: ineq for g}
\lim_{\delta\to 0}g_\delta(r)=r^{-x}\int_{\{b=r\}}|Rm|^{1-\epsilon}|\nabla b| \le \frac{C(n,r_0)}{(x-y)r_0^{y-x-1}}r^{-x+y}.
\end{equation}
Finally we observe that
\begin{align*}
\int_{b \ge r_0} |Rm|^2 |\nabla b|^2 &= \int_{r_0}^{\infty}\int_{b=s} |Rm|^2 |\nabla b| \,ds\\
&=\int_{r_0}^{\infty}\int_{b=s}|Rm|^{1+\epsilon}|Rm|^{1-\epsilon} |\nabla b| \,ds\\
&\le C_0^{1+\epsilon}\int_{r_0}^{\infty}s^{-2-2\epsilon}\int_{b=s}|Rm|^{1-\epsilon} |\nabla b| \,ds\\
 &\le \frac{C_0^{1+\epsilon}C(n,r_0)}{(x-y)r_0^{y-x-1}} \int_{r_0}^{\infty}s^{y-2-2\epsilon}\,ds,
\end{align*}
where \eqref{eq: ineq for g} was used in the last step. If $C_0$ is so small that $y=\frac{n}{2}-\sqrt{\left(\frac{n-2}{2}\right)^2-CC_0}<1+2\epsilon$, then the integral above is finite. Thus we conclude that $\int_M |Rm|^p < \infty$ for any $p \ge 2$ by H\"{o}lder's inequality. This proves the theorem.

\end{proof}

\end{document}